\pdfoutput=1
\documentclass[11pt,a4paper]{scrartcl}

\usepackage{iftex}
\PassOptionsToPackage{mathfont=stix2}{mystyle-koma-enhanced}
\usepackage{mystyle-koma-enhanced}
\usepackage{tikz}
\usetikzlibrary{arrows.meta}
\KOMAoptions{DIV=8}
\recalctypearea

\declaretheorem[style=mystyleplain,name=Theorem A,numbered=no]{mainthm}

\newcommand{\Ztwo}{\mathbb Z_2}

\title{The one-point Schreier Poisson boundary of Thompson's group \(F\)}
\author{Christian M\"onch\,\orcidlink{0000-0002-6531-6482}\thanks{Independent
Researcher, Darmstadt, Germany. Email:
\href{mailto:cmoench25@gmail.com}{\nolinkurl{cmoench25@gmail.com}}.}}
\date{}

\begin{document}
\maketitle

\begin{abstract}
We identify the Poisson boundary of the one-point Schreier-chain random walk
obtained by projecting the simple symmetric random walk on Thompson's group
\(F\) to the dyadic orbit point \(1/2\).  For the associated simple
labelled-generator walk on the dyadic Schreier graph, the full Poisson boundary
is the skeleton end boundary.  The proof combines the known description of
this Schreier graph as a binary-tree skeleton with recurrent one-dimensional
ray attachments with an explicit trace computation.  After tracing to the
grey skeleton and deleting holding probabilities, the walk becomes a reversible
nearest-neighbor walk on the rooted binary tree with two unequal classes of
edge conductance.  This reduces the boundary identification to standard
Poisson--Martin theory for transient walks on trees and leaves a finite
electrical-network calculation for the harmonic measure.

Following Kaimanovich's coding of skeleton ends by odd 2-adic integers
\cite[{\emph{Groups, Graphs and Random Walks}}, London Math. Soc. Lecture Note
Ser.~436, pp.~300--342, 2017]{Kaimanovich2017}, the hitting measure is a
biased Bernoulli product measure with explicitly computed bias.  It is singular
with respect to Haar measure, has full topological support, and is
exact-dimensional; these properties and the exact constants are proved here.

\end{abstract}



\section{Introduction}\label{sec:introduction}

Thompson's group \(F\) is a finitely generated group of dyadic
piecewise-linear homeomorphisms of the unit interval, and is a central test
object in geometric and measured group theory; see the introductory account of
Cannon, Floyd, and Parry~\cite{CannonFloydParry1996}. Its appeal comes from the
combination of an elementary definition with unusually rich large-scale and
dynamical behaviour. For example, Brown and Geoghegan~\cite{BrownGeoghegan1984}
used \(F\) to produce an infinite-dimensional torsion-free group of type
\(FP_\infty\). Thus \(F\) sits in a useful middle ground: concrete enough to
allow explicit calculations, but subtle enough that asymptotic questions about
walks, boundaries, and actions retain genuine content.

Kaimanovich~\cite{Kaimanovich2017} proved that every strictly non-degenerate
finitely supported random walk on Thompson's group \(F\) has nontrivial Poisson
boundary.  His construction uses the action on the dyadic orbit and gives
boundary factors from stabilizing logarithmic slope-jump configurations and
from limiting ends of Schreier graph trajectories.  Mishchenko
\cite{Mishchenko2015} studied the one-point action on dyadic numbers and proved
nontriviality for the corresponding boundary, while Stankov
\cite{Stankov2022Convergence} proved an end-convergence theorem for transient
induced walks on Schreier graphs.  These results locate the relevant boundary
phenomena.  The present paper identifies, in the basic one-point dyadic case,
the full Poisson boundary and its harmonic measure explicitly.

For the simple symmetric group walk \((W_n)\) with increments uniform on
\(\{x_0,x_0^{-1},x_1,x_1^{-1}\}\), each dyadic point \(q\) gives a projected
Schreier-chain walk \(Z_n^q=qW_n\).  Kaimanovich's group-level end boundary is
the coupled random section \(q\mapsto \lim_{n\to\infty} qW_n\), with all
coordinates driven by the same increments.  Here we compute the full Poisson
boundary of the single coordinate \(q=1/2\).  Thus the result is a one-point
Schreier-chain boundary, and should be viewed as a marginal of the
group-level end-section field rather than as an identification of the full
Poisson boundary of the group walk.

We work in Kaimanovich's real-line coordinate model and start the projected
walk at the grey root \(1/2\). The full Poisson boundary of this one-point
Markov chain is the skeleton end boundary, equivalently
\(\mathbb Z_2^\times\), with an explicit biased Bernoulli harmonic measure.

Let \(\Gamma\) denote the Schreier graph of the dyadic orbit in the standard
coordinate model. The vertex set is \(\mathbb Z[1/2]\), and the generator walk
uses the four maps

\[
\widetilde A,\quad \widetilde A^{-1},\quad
\widetilde B,\quad \widetilde B^{-1}
\]

with equal probability. The graph has a tree-like skeleton and recurrent
one-dimensional rays attached to it. After tracing the Markov chain to the
grey skeleton vertices

\[
B=\mathbb Z[1/2]\cap(0,1),
\]

and deleting holding probabilities, one obtains a reversible random walk on
the rooted binary tree \(B\). The root is \(1/2\), and the two child maps are

\[
f_0(\gamma)=\frac{\gamma}{2},
\qquad
f_1(\gamma)=\frac{\gamma}{2}+\frac12 .
\]

The resulting tree walk is the conductance walk with conductances \(1\) on
\(0\)-edges and \(1/2\) on \(1\)-edges; it is not the usual simple random walk
on the binary tree. This asymmetry gives the biased 2-adic harmonic measure.
\Cref{fig:skeleton-reduction} summarizes this reduction schematically.

\begin{figure}[htbp]
\centering
\begin{tikzpicture}[
  x=1cm,
  y=1cm,
  >=Stealth,
  skeleton/.style={line width=0.7pt},
  ray/.style={draw=gray, line width=0.5pt, densely dashed},
  vertex/.style={circle, draw, fill=white, inner sep=1.4pt, minimum size=4.6pt},
  root/.style={vertex, fill=gray!20},
  edge label/.style={font=\scriptsize, fill=white, inner sep=1pt},
  note/.style={font=\scriptsize, text=gray},
  boundary/.style={font=\scriptsize}
]
\node[root, label=above:{\(1/2\)}] (r) at (0,0) {};
\node[vertex, label=left:{\(1/4\)}] (r0) at (-2,-1.2) {};
\node[vertex, label=right:{\(3/4\)}] (r1) at (2,-1.2) {};
\node[vertex] (r00) at (-3,-2.4) {};
\node[vertex] (r01) at (-1,-2.4) {};
\node[vertex] (r10) at (1,-2.4) {};
\node[vertex] (r11) at (3,-2.4) {};
\node[vertex] (r000) at (-3.6,-3.45) {};
\node[vertex] (r001) at (-2.4,-3.45) {};
\node[vertex] (r110) at (2.4,-3.45) {};
\node[vertex] (r111) at (3.6,-3.45) {};

\draw[skeleton] (r) -- node[edge label, left] {\(f_0,\ c=1\)} (r0);
\draw[skeleton] (r) -- node[edge label, right] {\(f_1,\ c=\frac12\)} (r1);
\foreach \a/\b/\lab in {
  r0/r00/{\(c=1\)},
  r0/r01/{\(c=\frac12\)},
  r1/r10/{\(c=1\)},
  r1/r11/{\(c=\frac12\)}
}{
  \draw[skeleton] (\a) -- node[edge label] {\lab} (\b);
}
\draw[skeleton] (r00) -- (r000);
\draw[skeleton] (r00) -- (r001);
\draw[skeleton] (r11) -- (r110);
\draw[skeleton] (r11) -- (r111);

\draw[ray] (r0) -- ++(-0.65,0.35) coordinate (a1)
  -- ++(-0.45,0.25) coordinate (a2)
  -- ++(-0.35,0.2) coordinate (a3);
\draw[ray] (r10) -- ++(0.65,-0.2) coordinate (b1)
  -- ++(0.5,-0.18) coordinate (b2)
  -- ++(0.4,-0.14) coordinate (b3);
\draw[ray] (r01) -- ++(-0.55,-0.15) coordinate (c1)
  -- ++(-0.45,-0.15) coordinate (c2);
\foreach \p in {a1,a2,a3,b1,b2,b3,c1,c2} {
  \fill[gray] (\p) circle (1pt);
}
\node[note, align=center] at (-3.4,-0.35) {recurrent\\ray attachments};
\node[boundary] at (0,-4.2) {ends of \(B\): \(\partial B\cong\Ztwo^\times\)};
\draw[-{Stealth}, densely dotted] (r000) -- ++(-0.35,-0.45);
\draw[-{Stealth}, densely dotted] (r001) -- ++(0,-0.45);
\draw[-{Stealth}, densely dotted] (r110) -- ++(0,-0.45);
\draw[-{Stealth}, densely dotted] (r111) -- ++(0.35,-0.45);
\end{tikzpicture}
\caption{Schematic of the reduction. The solid binary tree is the grey
skeleton \(B\), while the dashed pieces indicate recurrent one-dimensional ray
attachments in the Schreier graph. Tracing the labelled-generator walk to
\(B\) and deleting holds gives the displayed conductance walk, whose ends are
coded by odd 2-adic integers.}
\label{fig:skeleton-reduction}
\end{figure}
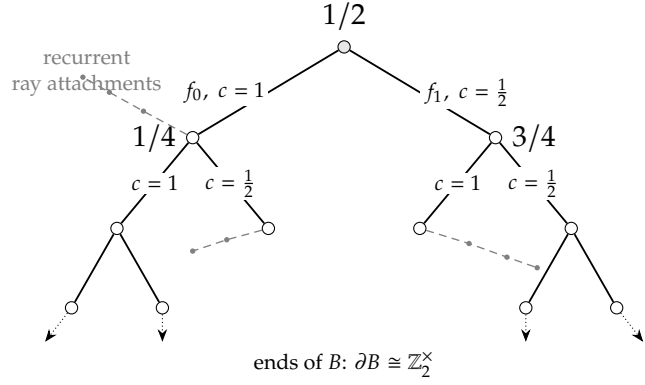

Once this reduction has been made, the identification of the Poisson boundary
with the end boundary is a standard consequence of boundary theory for
transient reversible walks on trees, or equivalently of the electrical-network
viewpoint on trees; see, for example, Woess~\cite{Woess2000,Woess2009}. The
point here is therefore not to introduce a new tree-boundary theorem. Rather,
the paper identifies the precise tree walk arising from the labelled
Schreier-chain dynamics of \(F\): the recurrent ray attachments have to be
removed without changing the Poisson boundary, the trace probabilities must be
computed with generator-label multiplicities, and the resulting conductances
\(1\) and \(1/2\) then determine the explicit 2-adic harmonic measure.

Throughout, the walk is the labelled-generator walk, not the unlabelled
graph-simple walk. If several generator labels give the same local move, they
are counted with their multiplicities. Thus the transition probabilities come
from the four maps
\(\widetilde A^{\pm1},\widetilde B^{\pm1}\), not from choosing uniformly among
the distinct neighboring vertices in the underlying unoriented graph.

The main theorem is the following.

\begin{mainthm}\label{thm:main}
Let \((X_n)\) be the simple labelled-generator walk on the dyadic
Schreier graph \(\Gamma\), started at the real-line coordinate vertex
\(1/2\). Then \(X_n\) converges almost surely to a skeleton end

\[
\Xi\in \partial\overline\Gamma\cong \mathbb Z_2^\times .
\]

The measured space \((\partial\overline\Gamma,\nu)\), where
\(\nu=\operatorname{Law}(\Xi)\), is the full Poisson boundary of this Markov
chain. Under the coding

\[
\Xi
\leftrightarrow
1+\sum_{j\ge 1}\varepsilon_j2^j
\in \mathbb Z_2^\times ,
\]

the digits \((\varepsilon_j)_{j\ge 1}\) are independent and identically
distributed with

\[
\mathbb P(\varepsilon_j=0)=2-\sqrt2,
\qquad
\mathbb P(\varepsilon_j=1)=\sqrt2-1 .
\]

Equivalently, for every cylinder

\[
[\varepsilon_1,\ldots,\varepsilon_k]
=
\left\{
1+\sum_{j\ge 1}\eta_j2^j:
\eta_i=\varepsilon_i,\ 1\le i\le k
\right\},
\]

one has

\[
\nu([\varepsilon_1,\ldots,\varepsilon_k])
=
(2-\sqrt2)^{N_0(\varepsilon_1,\ldots,\varepsilon_k)}
(\sqrt2-1)^{N_1(\varepsilon_1,\ldots,\varepsilon_k)} .
\]

Here \(N_i(\varepsilon_1,\ldots,\varepsilon_k)\) is the number of indices
\(j\le k\) with \(\varepsilon_j=i\).

\end{mainthm}

The qualification in the title is part of the statement: the theorem concerns
the induced Schreier-chain walk started at \(1/2\), not the full Poisson
boundary of the group random walk on \(F\).  Its role in the group-level
boundary picture is as a one-coordinate input.  The group walk couples all
dyadic coordinates and also carries slope-jump data; those coupled fields
belong to the next layer of the boundary problem, whereas the present paper
keeps the one-point Markov chain completely explicit and self-contained.

\section{Markov-chain boundary preliminaries}\label{sec:markov-chain-boundary-preliminaries}

Let \(P\) be a Markov kernel on a countable state space \(E\). A bounded
function \(h:E\to\mathbb R\) is \(P\)-harmonic if

\[
h(x)=\sum_{y\in E}P(x,y)h(y)
\]

for every \(x\in E\). Denote by \(H^\infty(E,P)\) the Banach space of
bounded \(P\)-harmonic functions, equipped with the supremum norm.

The standard harmonic-function characterization of the Poisson boundary will
be used: two Markov chains have isomorphic Poisson boundaries if their
bounded harmonic function spaces are isometrically order-isomorphic in a way
which preserves constants.

\subsection{Tracing to a recurrently visited subset}\label{subsec:tracing-to-a-recurrently-visited-subset}

The following elementary lemma is the basic probabilistic reduction.

\begin{lemma}\label{lem:trace-recurrent-subset}
Let \((X_n)\) be a Markov chain on a countable state space
\(E\), with transition kernel \(P\). Let \(A\subset E\) be nonempty. Assume
that, starting from every \(x\in E\), the chain hits \(A\) almost surely, and
that, starting from every \(a\in A\), it returns to \(A\) infinitely often
almost surely.

Define

\[
\tau_0=\inf\{n\ge 0:X_n\in A\},
\]

and, for \(k\ge 0\),

\[
\tau_{k+1}=\inf\{n>\tau_k:X_n\in A\}.
\]

The trace chain on \(A\) is

\[
Y_k=X_{\tau_k}
\]

with transition kernel

\[
P_A(a,b)=\mathbb P_a(X_{\tau_1}=b),
\qquad a,b\in A .
\]

Then restriction gives an isometric order-isomorphism

\[
H^\infty(E,P)\cong H^\infty(A,P_A).
\]

Consequently the Poisson boundaries of \((E,P)\) and \((A,P_A)\) are
isomorphic.

\end{lemma}
\begin{proof}
Let \(h\in H^\infty(E,P)\). Then \(h(X_n)\) is a bounded martingale.
Since \(\tau_1<\infty\) almost surely under \(\mathbb P_a\), optional
stopping applied to \(\tau_1\wedge n\), followed by dominated convergence,
gives, for \(a\in A\),

\[
h(a)=\mathbb E_a[h(X_{\tau_1})]
=\sum_{b\in A}P_A(a,b)h(b).
\]

Thus \(h|_A\in H^\infty(A,P_A)\).
The same optional-stopping argument at \(\tau_0\) gives, for every \(x\in E\),

\[
h(x)=\mathbb E_x[h(X_{\tau_0})].
\]

Indeed, apply optional stopping to \(h(X_{\tau_0\wedge n})\) and let
\(n\to\infty\).  Thus a bounded \(P\)-harmonic function is recovered from its
restriction to \(A\) by first hitting \(A\).

Conversely, let \(f\in H^\infty(A,P_A)\), and define

\[
h(x)=\mathbb E_x[f(X_{\tau_0})],
\qquad x\in E .
\]

This is well-defined because \(\tau_0<\infty\) almost surely and \(f\) is
bounded. If \(x\notin A\), the Markov property gives

\[
\sum_y P(x,y)h(y)=h(x).
\]

If \(a\in A\), then after one step the first subsequent hit of \(A\) is
\(X_{\tau_1}\), so

\[
\sum_y P(a,y)h(y)
=
\mathbb E_a[f(X_{\tau_1})]
=
\sum_{b\in A}P_A(a,b)f(b)
=f(a)
=h(a).
\]

Hence \(h\in H^\infty(E,P)\). Since \(\tau_0=0\) on \(A\), this extension
restricts to \(f\). The preceding paragraph shows that restriction followed by
extension recovers every bounded \(P\)-harmonic function. The two
constructions are therefore inverse to one another, and they preserve constants
and positivity. They are also isometries: restriction cannot increase the
supremum norm, while the extension satisfies

\[
\|h\|_{\infty,E}\le \|f\|_{\infty,A}
\]

by its definition as an expectation, and equality follows because \(h|_A=f\).
This proves the claim.
\end{proof}

\subsection{Removing holding probabilities}\label{subsec:removing-holding-probabilities}

The second reduction removes holding probabilities.

\begin{lemma}\label{lem:no-holding}
Let \(R\) be a Markov kernel on a countable set \(A\), and
assume \(R(a,a)<1\) for every \(a\in A\). Define the no-holding kernel
\(R^\circ\) by

\[
R^\circ(a,b)=\frac{R(a,b)}{1-R(a,a)}
\quad\text{for }b\ne a,
\qquad
R^\circ(a,a)=0 .
\]

Then

\[
H^\infty(A,R)=H^\infty(A,R^\circ).
\]

In particular, \((A,R)\) and \((A,R^\circ)\) have the same Poisson boundary.

\end{lemma}
\begin{proof}
The equation

\[
h(a)=R(a,a)h(a)+\sum_{b\ne a}R(a,b)h(b)
\]

is equivalent, since \(R(a,a)<1\), to

\[
h(a)=\sum_{b\ne a}\frac{R(a,b)}{1-R(a,a)}h(b).
\]

Thus the bounded harmonic functions are the same.
\end{proof}

\section{The dyadic Schreier graph}\label{sec:the-dyadic-schreier-graph}

We use Kaimanovich's coordinate model from Sections 2.C and 3.A of
\cite{Kaimanovich2017}. After the change of variables from the interval to the
real line, the standard generators \(A,B\) of \(F\) become piecewise-linear
maps \(\widetilde A,\widetilde B\) of \(\mathbb R\). The action is written on
the right, in postfix notation:

\[
\gamma.g=g(\gamma).
\]

Thus the product \(g_1g_2\) acts by first applying \(g_1\), then \(g_2\):

\[
(g_1g_2)(\gamma)=g_2(g_1(\gamma)).
\]

With this convention, the projected right random walk \(qW_n\) is obtained by
successively applying the increments on the right. The vertex set of the
Schreier graph is

\[
\mathbb D=\mathbb Z[1/2].
\]

The point \(1/2\) in this coordinate model is the grey root of the skeleton.
It should not be confused with the original interval point \(1/2\): under
Kaimanovich's change of variables, the latter corresponds to the real-line
coordinate \(0\). Starting from coordinate \(0\) gives the same boundary law
after the almost-sure first hit of \(B\) at the grey root \(1/2\).

The generators act by

\[
\widetilde A(\gamma)=\gamma-1,
\qquad
\widetilde A^{-1}(\gamma)=\gamma+1,
\]

and

\[
\widetilde B(\gamma)=
\begin{cases}
\gamma, & \gamma\le 0,\\
\gamma/2, & 0\le \gamma\le 2,\\
\gamma-1, & \gamma\ge 2,
\end{cases}
\]

with inverse

\[
\widetilde B^{-1}(\gamma)=
\begin{cases}
\gamma, & \gamma\le 0,\\
2\gamma, & 0\le \gamma\le 1,\\
\gamma+1, & \gamma\ge 1.
\end{cases}
\]

The simple labelled-generator walk on \(\Gamma\) applies one of

\[
\widetilde A,\quad \widetilde A^{-1},\quad
\widetilde B,\quad \widetilde B^{-1}
\]

with probability \(1/4\) at each step. This is the labelled-generator walk
induced by the symmetric measure on
\(\{\widetilde A,\widetilde A^{-1},\widetilde B,\widetilde B^{-1}\}\); it is
not the unlabelled graph-simple walk.

The required graph-theoretic facts are contained in Kaimanovich's description
of the Schreier graph in Section 5 of \cite{Kaimanovich2017}, based on
Savchuk's explicit model \cite{Savchuk2010,Savchuk2015}. The same geometry is
used in Mishchenko's proof \cite{Mishchenko2015} of nontriviality for the
one-point Poisson boundary.

\begin{proposition}[Kaimanovich--Savchuk ray decomposition {\cite[Section~5]{Kaimanovich2017}; \cite{Savchuk2010,Savchuk2015}}]\label{prop:ray-decomposition}
The Schreier graph \(\Gamma\) contains
two families of attached one-dimensional rays. The negative type rays are the
connected components of the subgraph on
\((-\infty,0]\cap\mathbb Z[1/2]\), attached to vertices in
\((0,1]\cap\mathbb Z[1/2]\). The positive type rays are the connected
components of the subgraph on \([2,\infty)\cap\mathbb Z[1/2]\), attached to
vertices in \([1,2)\cap\mathbb Z[1/2]\). The only vertex to which one ray of
each type is attached is \(1\).

\end{proposition}

\begin{proof}
This is Kaimanovich's Section 5 description of the graph \(\Gamma\), after the
coordinate change of Section 2.C.
\end{proof}

\begin{proposition}[Kaimanovich--Savchuk skeleton and grey tree {\cite[Section~5]{Kaimanovich2017}; \cite{Savchuk2010,Savchuk2015}}]\label{prop:skeleton-grey-tree}
Removing all negative and
positive type rays from \(\Gamma\) gives a tree \(\overline\Gamma\), the
skeleton. Its grey vertices are

\[
B=\mathbb Z[1/2]\cap(0,1).
\]

Rooted at \(o=1/2\), \(B\) is the binary tree with child maps

\[
f_0(\gamma)=\frac{\gamma}{2},
\qquad
f_1(\gamma)=\frac{\gamma}{2}+\frac12 .
\]

The full skeleton \(\overline\Gamma\) differs from this grey binary tree only
by the additional edge between \(1\) and \(1/2\) and by white vertices which
subdivide some grey-tree edges into two edges.

\end{proposition}

\begin{proof}
This is Kaimanovich's skeleton decomposition from Section 5.
\end{proof}

\begin{lemma}\label{lem:ray-excursions-recurrent}
For the labelled-generator walk
considered here, every excursion into an attached ray returns to its attachment
vertex almost surely.

\end{lemma}
\begin{proof}
On a negative type ray, \(\widetilde A\) and \(\widetilde A^{-1}\) give the
two nearest-neighbor moves, while \(\widetilde B\) and \(\widetilde B^{-1}\)
are holding moves. After holds are ignored this is the usual symmetric
nearest-neighbor walk on a half-line. On a positive type ray, away from the
finite attachment neighborhood, \(\widetilde A\) and \(\widetilde B\) move one
step toward the attachment point, while
\(\widetilde A^{-1}\) and \(\widetilde B^{-1}\) move one step away from it;
this is again the symmetric nearest-neighbor walk on a half-line. Such a walk
hits its origin almost surely, and adding holding probabilities or changing
finitely many transition probabilities near the origin does not change this
hitting probability.
\end{proof}

\begin{proposition}[Kaimanovich end-space coding {\cite[Section~5]{Kaimanovich2017}}]\label{prop:end-space-coding}
The end space of
\(\Gamma\) decomposes as

\[
\partial\Gamma
=
\partial_-\Gamma
\sqcup
\partial_+\Gamma
\sqcup
\partial\overline\Gamma,
\]

where \(\partial_-\Gamma\) and \(\partial_+\Gamma\) are the countable families
of negative and positive ray ends, and \(\partial\overline\Gamma\) is the end
space of the skeleton. The skeleton end space is identified with the
odd 2-adic integers \(\mathbb Z_2^\times\). Under the grey-tree coding,

\[
(\varepsilon_1,\varepsilon_2,\ldots)
\longmapsto
1+\sum_{j\ge 1}\varepsilon_j2^j .
\]

\end{proposition}

\begin{proof}
This is Kaimanovich's description of the space of ends in Section 5.
\end{proof}

\begin{lemma}\label{lem:centered-support-skeleton}
For the simple labelled-generator walk used in this paper, the limiting end has zero
probability to lie in \(\partial_-\Gamma\cup\partial_+\Gamma\). Thus the
end-boundary hitting measure is supported on \(\partial\overline\Gamma\).

\end{lemma}
\begin{proof}
To converge to the end of a fixed attached ray, the walk must
eventually remain in every tail of that ray. In particular, some excursion into
that ray would have to avoid its attachment vertex forever. By Lemma 3.3 this
has probability zero. There are only countably many attached rays, so the union
of these events has probability zero.
\end{proof}

\section{Reduction to a weighted binary-tree walk}\label{sec:reduction-to-a-weighted-binary-tree-walk}

It remains to compute the trace of the simple labelled-generator walk on the
grey binary tree

\[
B=\mathbb Z[1/2]\cap(0,1).
\]

Let \(\widehat P\) denote the trace kernel on \(B\), and let \(Q\) be the
no-holding kernel obtained from \(\widehat P\).

\subsection{Boundary preservation under the trace}\label{subsec:boundary-preservation-under-the-trace}

\begin{proposition}\label{prop:trace-preserves-boundary}
The Poisson boundary of the simple labelled-generator walk
\((\Gamma,P)\) is isomorphic to the Poisson boundary of the trace chain
\((B,\widehat P)\).

\end{proposition}
\begin{proof}
The recurrence hypothesis of Lemma 2.1 holds.
The components removed from the skeleton are one-dimensional negative and
positive type rays. On each such ray, away from the finite attachment
neighborhood, the simple labelled-generator walk is a symmetric
nearest-neighbor walk on a half-line, possibly with holding on negative rays.
Hence every ray excursion returns to its attachment point almost surely.

The remaining non-grey vertices in the skeleton are the white vertices in
\([1,2)\). Starting from such a vertex \(u\), Kaimanovich's first-hit
calculation for \(I=[0,1)\) gives almost-sure entrance into \(I\), with first
hit in \(\{u-1,u/2\}\). A first hit in \(B=(0,1)\) gives entrance into \(B\).
If the first hit is the endpoint \(0\), then the walk returns almost surely to
\(1\), and from \(1\) it hits \(1/2\in B\) almost surely, since excursions from
\(1\) into the adjacent rays return to \(1\) and the edge \(1\to1/2\) is
available at each return.

Therefore \(B\) is hit almost surely from every starting point of \(\Gamma\).
Now start from \(a\in B\). After the walk leaves \(B\), the preceding
classification applies to the new state, so the next hit of \(B\) is finite
almost surely. Applying the strong Markov property at successive hitting times
gives \(\tau_k<\infty\) for every \(k\). Thus a walk started in \(B\) returns
to \(B\) infinitely often almost surely, and Lemma 2.1 applies.
\end{proof}

\subsection{Computation of the trace kernel}\label{subsec:computation-of-the-trace-kernel}

Relabel \(B\) by binary words. Put

\[
x_\varnothing=\frac12,
\qquad
x_{w0}=\frac{x_w}{2},
\qquad
x_{w1}=\frac{x_w}{2}+\frac12 .
\]

Thus \(B\cong\{0,1\}^{<\mathbb N}\). Write \(w^-\) for the parent of a
non-root word \(w\).

The local geometry is as follows. The edge from \(\gamma\) to
\(f_0(\gamma)=\gamma/2\) is direct, since it is obtained by applying
\(\widetilde B\). The edge from \(\gamma\) to
\(f_1(\gamma)=\gamma/2+1/2\) passes through the white vertex \(\gamma+1\):

\[
\gamma
\xrightarrow{\widetilde A^{-1}}
\gamma+1
\xrightarrow{\widetilde B}
\frac{\gamma+1}{2}
=
\frac{\gamma}{2}+\frac12 .
\]

Similarly, a \(0\)-child is directly connected to its parent by
\(\widetilde B^{-1}\), while a \(1\)-child is connected to its parent through
a white vertex.

Two elementary first-hit facts are used. First, on negative type rays the walk
is a lazy symmetric nearest-neighbor walk on a half-line; on positive type rays
it is a symmetric nearest-neighbor walk, with two generator labels in each
direction away from the finite attachment neighborhood. Consequently every
excursion into a negative or positive type ray returns almost surely to its
attachment vertex.

Second, Kaimanovich's computation of the induced walk on \(I=[0,1)\) implies
that, starting from any white vertex \(u\in[1,2)\), the walk hits \(I\) almost
surely, and the first hit is supported on the two points

\[
u-1,
\qquad
\frac{u}{2},
\]

with equal probabilities \(1/2\) and \(1/2\). Thus a white vertex between two
grey vertices sends the trace to either grey endpoint with probability
\(1/2\), except at the endpoint \(u=1\), where the two first hits in \(I\) are
\(0\) and \(1/2\). For the trace to \(B=(0,1)\), both \(0\) and \(1\) return
to \(1/2\) almost surely through recurrent ray excursions and the direct
skeleton edge \(1\leftrightarrow1/2\).

\begin{proposition}\label{prop:trace-kernel}
The no-holding trace \(Q\) is the reversible
nearest-neighbor random walk on the rooted binary tree \(B\) with conductances

\[
c(w,w0)=1,
\qquad
c(w,w1)=\frac12 .
\]

Equivalently, the edge resistances are

\[
r(w,w0)=1,
\qquad
r(w,w1)=2 .
\]

More explicitly, at the root,

\[
Q(\varnothing,0)=\frac23,
\qquad
Q(\varnothing,1)=\frac13 .
\]

If \(w\ne\varnothing\) ends in \(0\), then

\[
Q(w,w^-)=\frac25,
\qquad
Q(w,w0)=\frac25,
\qquad
Q(w,w1)=\frac15 .
\]

If \(w\ne\varnothing\) ends in \(1\), then

\[
Q(w,w^-)=\frac14,
\qquad
Q(w,w0)=\frac12,
\qquad
Q(w,w1)=\frac14 .
\]

The trace probabilities before deleting holds are:

\[
\begin{array}{c|c|c|c|c}
\text{state } \gamma
& \widehat P(\gamma,\gamma)
& \widehat P(\gamma,p(\gamma))
& \widehat P(\gamma,\gamma_0)
& \widehat P(\gamma,\gamma_1)
\\ \hline
\text{non-root }0\text{-child}
& 3/8 & 1/4 & 1/4 & 1/8
\\
\text{non-root }1\text{-child}
& 1/2 & 1/8 & 1/4 & 1/8
\\
o=\varnothing
& 5/8 & - & 1/4 & 1/8
\end{array}
\]

Here \(p(\gamma)\) denotes the parent of a non-root vertex, and
\(\gamma_i=f_i(\gamma)\). The no-holding kernel \(Q\) is obtained by
normalizing each row after removing the holding column.

\end{proposition}
\begin{proof}
Compute \(\widehat P\) before deleting holds. Let
\(\gamma\in B\), and set

\[
\gamma_0=f_0(\gamma),
\qquad
\gamma_1=f_1(\gamma).
\]

Every non-root grey vertex lies in exactly one of the two intervals
\((0,1/2)\) and \((1/2,1)\). These are precisely the \(0\)-children and
\(1\)-children, respectively. Indeed, if \(0<\gamma<1/2\), then its parent is
\(2\gamma\), while if \(1/2<\gamma<1\), then its parent is \(2\gamma-1\).

First assume \(0<\gamma<1/2\). Then \(\gamma\) is a \(0\)-child of its parent,
so

\[
p(\gamma)=2\gamma.
\]

The four generator moves behave as follows:

\begin{itemize}[leftmargin=*]
\item \(\widetilde A(\gamma)=\gamma-1\) enters a negative recurrent ray and
  returns to \(\gamma\) before the next hit of \(B\).
\item \(\widetilde A^{-1}(\gamma)=\gamma+1\) enters the white vertex between
  \(\gamma\) and \(\gamma_1\), from which the trace hits these two grey
  vertices with probabilities \(1/2\) and \(1/2\).
\item \(\widetilde B(\gamma)=\gamma_0\).
\item \(\widetilde B^{-1}(\gamma)=2\gamma=p(\gamma)\).
\end{itemize}

Therefore

\[
\widehat P(\gamma,\gamma)=\frac14+\frac14\cdot\frac12=\frac38,
\]

\[
\widehat P(\gamma,\gamma_0)=\frac14,
\qquad
\widehat P(\gamma,\gamma_1)=\frac14\cdot\frac12=\frac18,
\qquad
\widehat P(\gamma,p(\gamma))=\frac14 .
\]

Deleting the hold gives

\[
Q(\gamma,p(\gamma))=\frac{1/4}{5/8}=\frac25,
\qquad
Q(\gamma,\gamma_0)=\frac{1/4}{5/8}=\frac25,
\qquad
Q(\gamma,\gamma_1)=\frac{1/8}{5/8}=\frac15 .
\]

Now assume \(1/2<\gamma<1\). Then \(\gamma\) is a \(1\)-child and its parent is

\[
p(\gamma)=2\gamma-1.
\]

The move \(\widetilde B^{-1}(\gamma)=2\gamma\) enters the white vertex between
\(p(\gamma)\) and \(\gamma\), so it contributes half to holding and half to
the parent. The same local computation gives

\[
\widehat P(\gamma,\gamma)=\frac12,
\qquad
\widehat P(\gamma,\gamma_0)=\frac14,
\qquad
\widehat P(\gamma,\gamma_1)=\frac18,
\qquad
\widehat P(\gamma,p(\gamma))=\frac18 .
\]

Deleting holds gives

\[
Q(\gamma,p(\gamma))=\frac14,
\qquad
Q(\gamma,\gamma_0)=\frac12,
\qquad
Q(\gamma,\gamma_1)=\frac14 .
\]

At the root \(o=1/2\), the four generator moves are:

\begin{itemize}[leftmargin=*]
\item \(\widetilde A(o)=-1/2\), which enters the negative type ray attached to
  \(o\) and returns to \(o\) almost surely before the next hit of \(B\).
\item \(\widetilde A^{-1}(o)=3/2\), a white vertex whose first hit of \(I\) is
  \(o\) or \(3/4=o1\), each with probability \(1/2\).
\item \(\widetilde B(o)=1/4=o0\).
\item \(\widetilde B^{-1}(o)=1\). From \(1\), the walk hits \(o\) almost surely
  before any other point of \(B\): excursions from \(1\) to \(0\) or into the
  positive type ray are recurrent, and the edge \(1\to o\) is available at
  each return to \(1\).
\end{itemize}

Therefore

\[
\widehat P(o,o)=\frac58,
\qquad
\widehat P(o,o0)=\frac14,
\qquad
\widehat P(o,o1)=\frac18,
\]

and therefore

\[
Q(o,o0)=\frac{1/4}{3/8}=\frac23,
\qquad
Q(o,o1)=\frac{1/8}{3/8}=\frac13 .
\]

These probabilities are those of the stated conductance walk. If
\(w\) ends in \(0\), then the parent edge has conductance \(1\), while the two
child edges have conductances \(1\) and \(1/2\). Normalizing gives

\[
\frac{1}{1+1+1/2}=\frac25,
\qquad
\frac{1/2}{1+1+1/2}=\frac15 .
\]

If \(w\) ends in \(1\), then the parent edge has conductance \(1/2\), while
the two child edges have conductances \(1\) and \(1/2\). Normalizing gives

\[
\frac{1/2}{1/2+1+1/2}=\frac14,
\qquad
\frac{1}{1/2+1+1/2}=\frac12 .
\]

The root has incident conductances \(1\) and \(1/2\), giving probabilities
\(2/3\) and \(1/3\).
\end{proof}

Proposition 4.1, Lemma 2.2, and Proposition 4.2 give

\[
\operatorname{Pois}(\Gamma,P)
\cong
\operatorname{Pois}(B,Q).
\]

\section{The weighted binary tree}\label{sec:the-weighted-binary-tree}

Let \(T=\{0,1\}^{<\mathbb N}\), rooted at \(\varnothing\), be the word model
of the grey binary tree \(B\) under the relabeling \(x_w\leftrightarrow w\).
The nearest-neighbor random walk \(Q\) has conductances

\[
c(w,w0)=1,
\qquad
c(w,w1)=\frac12 .
\]

The boundary is

\[
\partial T=\{0,1\}^{\mathbb N}.
\]

For \(\xi=(\xi_1,\xi_2,\ldots)\in\partial T\), the corresponding ray is

\[
\varnothing,\ \xi_1,\ \xi_1\xi_2,\ldots .
\]

\subsection{Transience and convergence to an end}\label{subsec:transience-and-convergence-to-an-end}

Let \(R\) be the effective resistance from a vertex to infinity through its
descendant subtree. By self-similarity, \(R\) satisfies

\[
\frac1R
=
\frac1{1+R}
+
\frac1{2+R}.
\]

Equivalently,

\[
(1+R)(2+R)=R(3+2R),
\]

and hence

\[
R^2=2.
\]

Thus \(R=\sqrt2\).

Equivalently, let \(R_n\) be the effective resistance from a
vertex to generation \(n\) below it, with all generation-\(n\) vertices wired
together. Then \(R_0=0\), and

\[
R_{n+1}
=
\left(
\frac1{1+R_n}
+
\frac1{2+R_n}
\right)^{-1}
=
\frac{(1+R_n)(2+R_n)}{3+2R_n}.
\]

If \(F(r)=(1+r)(2+r)/(3+2r)\), then

\[
F(r)-r=\frac{2-r^2}{3+2r}.
\]

For \(0\le r\le\sqrt2\), also

\[
\sqrt2-F(r)
=
\frac{(\sqrt2-r)(r+3-\sqrt2)}{3+2r}.
\]

Hence \(0\le R_n\le R_{n+1}\le\sqrt2\), and \(R_n\) increases to the fixed
point \(\sqrt2\). In particular, the effective resistance from the root to
infinity is finite, so the conductance walk is transient.

\begin{proposition}\label{prop:tree-convergence}
The \(Q\)-walk on \(T\), started from any vertex,
converges almost surely to an end \(\Xi\in\partial T\).

\end{proposition}
\begin{proof}
Since the walk is transient, every finite set of vertices is visited
only finitely often almost surely. Fix \(k\ge 0\). After the final visit to
the finite ball of radius \(k\) around the root, the path remains forever in
one component below level \(k\). Thus the first \(k\) digits of the limiting
ray eventually stabilize. Since this holds for every \(k\), the path
determines a unique end of \(T\), and the walk converges to that end in the end
compactification.
\end{proof}

\subsection{Harmonic measure}\label{subsec:harmonic-measure}

Let \(\nu\) be the law of the limiting end \(\Xi\) for the walk started at the
root.

The forward \(0\)-branch from any vertex has total resistance

\[
1+R=1+\sqrt2,
\]

and the forward \(1\)-branch has total resistance

\[
2+R=2+\sqrt2.
\]

The corresponding effective conductances are

\[
C_0=\frac1{1+\sqrt2},
\qquad
C_1=\frac1{2+\sqrt2}.
\]

The following standard branch-splitting fact for transient conductance walks on
trees will be used. Remove a vertex \(w\), and let \(\mathcal T_i\) be some of
the incident components, with neighbouring vertices \(v_i\). Let \(C_i\) be the
effective conductance from \(w\) to infinity through \(\mathcal T_i\), including
the first edge \(w\sim v_i\). Conditional on the limiting end lying in one of
the components \(\mathcal T_i\), the probability that it lies in
\(\mathcal T_i\) is

\[
\frac{C_i}{\sum_j C_j}.
\]

Indeed, let \(c_i=c(w,v_i)\). By the last-exit decomposition, the probability
of a path whose final departure from \(w\) into the listed components is along
\(w\to v_i\) is proportional to

\[
c_i\,\mathbb P_{v_i}(T_w=\infty).
\]

The quantity \(c_i\,\mathbb P_{v_i}(T_w=\infty)\) is the effective conductance
from \(w\) to infinity through \(\mathcal T_i\). This follows either by the
Dirichlet principle or by applying the same identity in finite truncations of
the branch and then passing to the limit. Therefore the branch probabilities
are proportional to the effective conductances \(C_i\). At the root the listed
components are exactly the two forward branches, so this gives the unconditional
first-digit distribution. At a non-root vertex it gives the corresponding
conditional split inside the cylinder determined by that vertex.

Applying this to the two forward branches at any vertex gives the same split:

\[
p_0=\frac{C_0}{C_0+C_1}
=
\frac{(1+\sqrt2)^{-1}}
{(1+\sqrt2)^{-1}+(2+\sqrt2)^{-1}}
=
2-\sqrt2,
\]

and

\[
p_1=\frac{C_1}{C_0+C_1}
=
\sqrt2-1.
\]

Equivalently, in the finite wired truncation used above, the two first-branch
probabilities are

\[
p_{0,n}
=
\frac{(1+R_n)^{-1}}{(1+R_n)^{-1}+(2+R_n)^{-1}}
=
\frac{2+R_n}{3+2R_n},
\]

and

\[
p_{1,n}
=
\frac{(2+R_n)^{-1}}{(1+R_n)^{-1}+(2+R_n)^{-1}}
=
\frac{1+R_n}{3+2R_n}.
\]

Since \(R_n\to\sqrt2\), these finite probabilities converge to \(p_0\) and
\(p_1\).

\begin{proposition}\label{prop:harmonic-measure-product}
Under the coding \(\partial T=\{0,1\}^{\mathbb N}\), the
harmonic measure \(\nu\) is the Bernoulli product measure with weights

\[
p_0=2-\sqrt2,
\qquad
p_1=\sqrt2-1.
\]

That is, for every cylinder

\[
[\varepsilon_1,\ldots,\varepsilon_k]
=
\{\xi\in\partial T:\xi_j=\varepsilon_j,\ 1\le j\le k\},
\]

one has

\[
\nu([\varepsilon_1,\ldots,\varepsilon_k])
=
\prod_{j=1}^k p_{\varepsilon_j}.
\]

\end{proposition}
\begin{proof}
The length-one cylinder probabilities are the root branch probabilities
computed above. Since every descendant subtree is isomorphic to the whole
rooted conductance network, the conditional branch-splitting fact applies at
the last visit to the vertex \(w\) on the event \([w]\). Thus

\[
\nu([w0]\mid [w])=p_0,
\qquad
\nu([w1]\mid [w])=p_1.
\]

Induction on \(|w|\) gives the cylinder formula.
\end{proof}

These probabilities are not first-step probabilities. At the root,

\[
Q(\varnothing,0)=\frac23,
\qquad
Q(\varnothing,1)=\frac13,
\]

whereas

\[
\nu([0])=2-\sqrt2,
\qquad
\nu([1])=\sqrt2-1.
\]

Backtracking before the eventual last exit changes the weights.

\subsection{Maximality of the end boundary}\label{subsec:maximality-of-the-end-boundary}

The end boundary is maximal.

\begin{proposition}\label{prop:end-boundary-maximal}
The measured end boundary \((\partial T,\nu)\) is the full
Poisson boundary of the weighted tree walk \((T,Q)\).

\end{proposition}
\begin{proof}
By the Martin-boundary theorem for random walks on trees
\cite{PicardelloWoess1987}, see also \cite[Chapter 9.C]{Woess2009}, the Martin
boundary at \(t=1\) of a transient uniformly irreducible bounded-step random
walk on a tree is the space of ends, and this boundary is minimal.

The hypotheses are satisfied. The tree \(T\) is locally finite, \(Q\) is
nearest-neighbor and irreducible, transience was proved in Proposition 5.1, and
the nonzero transition probabilities are uniformly bounded below:

\[
Q(x,y)\ge \frac15
\qquad
\text{whenever }x\sim y .
\]

Thus the minimal Martin boundary of \((T,Q)\) is \(\partial T\).

The Poisson--Martin representation theorem
\cite[Chapter 7.D--E, Theorems 7.53 and 7.61]{Woess2009} represents every
positive harmonic function uniquely by a measure on the minimal Martin
boundary. Let \(o\) be the root, let \(\nu_o=\nu\) be the root hitting measure,
and write

\[
K(x,\xi)=\frac{d\nu_x}{d\nu_o}(\xi)
\]

for the Martin kernel normalized by \(K(o,\xi)=1\). The constant harmonic
function \(1\) is represented by \(\nu_o\), because

\[
1=\int_{\partial T}K(x,\xi)\,d\nu_o(\xi)
  =\int_{\partial T}1\,d\nu_x(\xi).
\]

Now let \(h\) be a positive harmonic function with \(0\le h\le1\). Let
\(\mu_h\) be its Martin representing measure. Since \(1-h\) is also positive
harmonic, uniqueness of the minimal Martin representation gives

\[
\mu_h+\mu_{1-h}=\nu_o.
\]

Hence \(\mu_h\ll\nu_o\), say \(d\mu_h=\varphi\,d\nu_o\), with
\(0\le\varphi\le1\). Therefore

\[
h(x)=\int_{\partial T}\varphi(\xi)\,d\nu_x(\xi)
\]

with \(\varphi\in L^\infty(\partial T,\nu_o)\). By linearity, the same
conclusion holds for every bounded real harmonic function: apply the preceding
argument to
\((h+M)/(2M)\), where \(M>\|h\|_\infty\), and rearrange. Complex-valued
functions follow by applying this to real and imaginary parts. Since Proposition
\ref{prop:tree-convergence} gives almost sure convergence to an end, the
\(\sigma\)-field generated by the limiting end represents all bounded harmonic
functions. Thus \((\partial T,\nu_o)\) is the full Poisson boundary.
\end{proof}

\section{The 2-adic formulation}\label{sec:the-2-adic-formulation}

The end boundary of the grey binary tree \(B\) is identified with the odd
2-adic integers as follows. The downward maps are

\[
f_\varepsilon(\gamma)=\frac{\gamma}{2}+\varepsilon\frac12,
\qquad
\varepsilon\in\{0,1\}.
\]

Starting at \(\gamma_1=1/2\), a ray with digits
\((\varepsilon_1,\varepsilon_2,\ldots)\) satisfies

\[
\gamma_{n+1}=f_{\varepsilon_n}(\gamma_n).
\]

Then \(2^n\gamma_n\) converges in the 2-adic topology to

\[
1+\sum_{j\ge 1}\varepsilon_j2^j.
\]

The end boundary is therefore identified with
\(\mathbb Z_2^\times\), the odd 2-adic integers, by

\[
\Phi:\partial T\to\mathbb Z_2^\times,
\qquad
\Phi(\varepsilon_1,\varepsilon_2,\ldots)
=
1+\sum_{j\ge 1}\varepsilon_j2^j .
\]

The length-\(k\) cylinder \([\varepsilon_1,\ldots,\varepsilon_k]\) corresponds
to the congruence class

\[
a\equiv
1+\sum_{j=1}^k\varepsilon_j2^j
\pmod{2^{k+1}}
\]

inside \(\mathbb Z_2^\times\).

The first \(0\)-branch

\[
\frac12\longmapsto \frac14
\]

corresponds to the congruence class \(1\pmod 4\), while the first \(1\)-branch

\[
\frac12\longmapsto \frac34
\]

corresponds to the congruence class \(3\pmod 4\).

The original walk on \(\Gamma\) converges to the same skeleton end as its
trace on \(B\). Indeed, no attached ray end can be the limit, because every
excursion into a negative or positive type ray returns to its attachment point
almost surely. Let \(\omega\in\partial\overline\Gamma\) be the limiting end of
the trace. For any finite set \(K\subset\Gamma\), all sufficiently late trace
vertices lie in the component of \(\Gamma\setminus K\) corresponding to
\(\omega\), and avoid the finitely many attachment vertices whose rays meet
\(K\). The ray excursions attached to those late trace vertices therefore also
lie in the same component. Hence the full path converges to \(\omega\) in the
end compactification of \(\Gamma\).

Propositions 4.1, 4.2, 5.1, 5.2, and 5.3 prove Theorem A.

\section{Consequences for harmonic measure}\label{sec:consequences-for-harmonic-measure}

\subsection{Singularity with respect to Haar measure}\label{subsec:singularity-with-respect-to-haar-measure}

\begin{corollary}\label{cor:singularity-support}
The harmonic measure \(\nu\) is singular with respect to
Haar probability measure on \(\mathbb Z_2^\times\). It has full topological
support.

\end{corollary}
\begin{proof}
Let \(m_{\operatorname{Haar}}\) denote Haar probability measure on
\(\mathbb Z_2^\times\). Under the digit coding above, Haar measure is fair
product measure on the binary digits \((\varepsilon_j)_{j\ge 1}\). Our
harmonic measure is the biased product measure with

\[
\mathbb P(\varepsilon_j=1)=\sqrt2-1\ne \frac12 .
\]

By the law of large numbers,

\[
\nu\text{-a.e.}\quad
\lim_{k\to\infty}
\frac1k
\#\{1\le j\le k:\varepsilon_j=1\}
=
\sqrt2-1,
\]

whereas

\[
m_{\operatorname{Haar}}\text{-a.e.}\quad
\lim_{k\to\infty}
\frac1k
\#\{1\le j\le k:\varepsilon_j=1\}
=
\frac12 .
\]

These two full-measure sets are disjoint. Hence

\[
\nu\perp m_{\operatorname{Haar}}.
\]

Since both digit weights \(p_0,p_1\) are positive, \(\nu\) has full
topological support on \(\mathbb Z_2^\times\).
\end{proof}

\subsection{Exact dimension}\label{subsec:exact-dimension}

\begin{corollary}\label{cor:dimension}
Equip \(\mathbb Z_2^\times\) with the standard 2-adic
metric \(d_2(x,y)=|x-y|_2\). The harmonic measure \(\nu\) is
exact-dimensional, and its pointwise dimension is

\[
\dim\nu
=
\frac{
-(2-\sqrt2)\log(2-\sqrt2)
-(\sqrt2-1)\log(\sqrt2-1)
}{\log 2}
\approx 0.9786600844 .
\]

Equivalently, this is the Hausdorff dimension of the measure \(\nu\).

\end{corollary}
\begin{proof}
A length-\(k\) cylinder in the digit coordinates has 2-adic diameter
\(2^{-(k+1)}\), so replacing \(2^{-(k+1)}\) by \(2^{-k}\) does not change any
dimension limit. For \(\nu\)-almost every \(\xi=(\varepsilon_j)\),

\[
\log\nu([\varepsilon_1,\ldots,\varepsilon_k])
=
N_0(k)\log p_0+N_1(k)\log p_1 .
\]

By the law of large numbers,

\[
\frac{N_0(k)}{k}\to p_0,
\qquad
\frac{N_1(k)}{k}\to p_1 .
\]

Therefore

\[
\lim_{k\to\infty}
\frac{\log\nu([\varepsilon_1,\ldots,\varepsilon_k])}
{\log 2^{-k}}
=
\frac{-p_0\log p_0-p_1\log p_1}{\log 2}.
\]

With

\[
p_0=2-\sqrt2,
\qquad
p_1=\sqrt2-1,
\]

this gives the stated pointwise dimension. Since the pointwise dimension exists
and is constant \(\nu\)-almost everywhere, \(\nu\) is exact-dimensional, and
its measure dimension is this common value.
\end{proof}

\begin{remark}[Other starting vertices]\label{rem:other-starts}
For each starting vertex \(x\in T\), let

\[
\nu_x=\operatorname{Law}_x(\Xi),
\qquad x\in T.
\]

For a nearest-neighbor transient random walk on a tree, the Martin kernel
extends continuously to the end boundary. In particular,

\[
\frac{d\nu_x}{d\nu_o}(\xi)=K(x,\xi),
\]

where \(K\) is the Martin kernel. Moreover, \(K(x,\xi)\) depends on \(\xi\)
only through the confluent \(x\wedge \xi\). Hence \(K(x,\cdot)\) is locally
constant on the finite partition of \(\partial T\) determined by the shadows
seen from \(x\). This gives mutual absolute continuity of all starting-point
harmonic measures.

\end{remark}

\begin{remark}[Group-level interface]\label{rem:group-level-interface}
Theorem A is intentionally one-coordinate.  For the group random
walk on \(F\), the end data form the coupled section
\[
q\longmapsto \lim_{n\to\infty} qW_n,
\qquad q\in\mathbb Z[1/2],
\]
with all coordinates driven by the same increments.  Kaimanovich's boundary
construction also includes the stabilizing slope-jump field of \(W_n\).  The
present paper supplies the exact one-point boundary, trace kernel, harmonic
measure, and ray-return inputs for that larger picture, but it makes no
maximality claim for the coupled group-level field.
\end{remark}

\nocite{CartwrightSoardiWoess1993}
\printbibliography

\end{document}